\documentclass[12pt]{amsart}

\usepackage{amsfonts}

\usepackage{mathtext}
\usepackage[cp1251]{inputenc}
\usepackage[T2A]{fontenc}
\RequirePackage{srcltx}
\usepackage[english]{babel}

\usepackage[dvips]{graphicx}
\usepackage{amsmath}
\usepackage{amssymb}
\usepackage{amsxtra}
\usepackage{latexsym}
\usepackage{ifthen}
\usepackage{mathrsfs}

\numberwithin{equation}{section} 
\newcommand{\BR}{\mathbb{{R}}}

\newtheorem{theorem}{Theorem}[section]
\newtheorem{proposition}[theorem]{Proposition}
\newtheorem{corollary}[theorem]{Corollary}
\theoremstyle{rem}
\newtheorem{rem}[theorem]{Remark}
\newtheorem{lemma}[theorem]{Lemma}
\newtheorem*{thma}{Theorem A}
\newtheorem*{thmb}{Theorem B}
\newtheorem*{thmc}{Theorem C}

\usepackage{mathrsfs}

\theoremstyle{definition}

\newtheorem{definition}[theorem]{Definition}

\def \l{\lambda}
\def \e{\varepsilon}

\def \a{\alpha}

\def \d{\delta}
\def \D{\Delta}
\def \g{\gamma}
\def \t{\tau}

\def \s{\sigma}
\def \b{\beta}

\numberwithin{equation}{section}

\def\N{{{\Bbb N}}}
\def\Z{{{\Bbb Z}}}

\def\R{{\Bbb R}}

\def\l{{\lambda }}
\def\a{{\alpha }}
\def\D{{\Delta }}
\def\t{{\theta }}

\def\a{{\alpha}}
\def\b{{\beta}}
\def\d{{\delta}}
\def\e{{\varepsilon}}
\def\s{{\sigma}}
\def\g{{\gamma}}
\def\vp{{\varphi}}

\def\supp{\operatorname{supp}}

\def\){\right)}
\def\({\left(}

\par

\bigskip
\bigskip

\begin{document}

\author{Yu. Kolomoitsev and E. Liflyand}

\title[Fourier integrals] {On weighted conditions for the absolute convergence of Fourier integrals}



\address{Universit\"at zu L\"ubeck,
Institut f\"ur Mathematik,
Ratzeburger Allee 160,
23562 L\"ubeck}
\email{kolomoitsev@math.uni-luebeck.de, kolomus1@mail.ru}
\address{Department of Mathematics, Bar-Ilan University, 52900 Ramat-Gan, Israel}
\email{liflyand@math.biu.ac.il}

\begin{abstract}
In this paper we obtain new sufficient conditions for representation of a function as an
absolutely convergent Fourier integral. Unlike those known earlier,
these conditions are given in terms of belonging to
weighted spaces. Adding weights allows one to extend the range of
application of such results to Fourier multipliers with
unbounded  derivatives.
\end{abstract}

\maketitle

\section{Introduction}

If, for $n=1,2,....$
\begin{eqnarray}\label{belw}
f(x)=\int_{\BR^n}g(t)e^{i(x,t)}dt,\qquad g\in L_1(\BR^n),
\end{eqnarray}
where $(x,t)=x_1t_1+...+x_nt_n,$ we say that $f$ belongs to Wiener's algebra $W_0(\BR^n),$ written $f\in W_0(\BR^n),$
with $\|f\|_{W_0}=\|g\|_{L_1(\BR^n)}.$ Wiener's algebra is an important class of functions and its in-depth
study is motivated both by many points of interest in the topic itself and by its relations to other
areas of analysis, such as Fourier multipliers or comparison of differential operators. The history, motivations
and various conditions of belonging to Wiener's algebra are overviewed in detail in a recent survey paper \cite{LST}.

Of course, \cite{LST} summarized the long term studies of many mathematicians and gave a comprehensive picture of the subject.
However, these studies are continuing, see, e.g., \cite{K14}, \cite{KL}, \cite{Li},  \cite{LiOu}--\cite{LiTr1}, \cite{TrUMZ}. In these
works the undertaken efforts have mainly been aimed at obtaining conditions of mixed type, in the sense that conditions are posed
simultaneously on the function and its derivatives.

Naturally, certain such conditions were known earlier, see, e.g., \cite{B}, \cite{Lofstrom}. In the latter, the (multidimensional)
Riesz fractional differentiation is defined by
$$
(-\Delta)^ {\frac{\alpha}2} f =\mathscr{F}^{-1} |x|^\a
\mathscr{F}f,\qquad\alpha>0,
$$
where $\mathscr{F}$ means the Fourier operator, while
$\Delta=\sum_{j=1}^n\frac{\partial^2}{\partial x_j^2}$ denotes the
Laplace operator, and the result reads as follows (see~\cite{Lofstrom}).
\begin{thma}\label{Lo}
Let $f\in L_2(\mathbb R^n)$, and $(-\Delta)^ {\frac{\alpha}2} f\in
L_2(\mathbb R^n)$, $\alpha>\frac{n}2$, then $\mathscr{F}f\in
L_1(\mathbb R^n)$.
\end{thma}
To give further convenient formulations in the multivariate case, we need
additional notations. Let $\eta$ be an $n$-dimensional vector with
the entries either $0$ or $1$ only. Here and in what follows, $D^\eta
f$ for $\eta={\bf 0}=(0,0,...,0)$ or $\eta={\bf 1}=(1,1,...,1)$ mean
the function itself and the mixed derivative in each variable,
respectively, where
\begin{eqnarray*} D^\eta f(x)=\left(\prod\limits_{j:\, \eta_j=1}
\frac{\partial}{\partial x_j}\right)f(x).          \end{eqnarray*}

One of the multidimensional results we are going to generalize reads as follows (see~\cite{Samko}).

\begin{thmb} Let $f\in L_1(\mathbb R^n).$ If all the
mixed derivatives (in the distributional sense) $D^\eta f(x)\in
L_p(\mathbb R^n),$ $\eta\ne{\bf 0},$ where $1<p\le2,$ then $f\in W_0(\mathbb R^n).$
\end{thmb}

However, the main motivation to our work was given by the following recent result \cite{Li}:
if $f\in L^p(\mathbb R)$, $1\le p<\infty,$ and $f'\in L^q(\mathbb R)$, $1<q<\infty$,
for $p$ and $q$ such that
\begin{eqnarray}
\label{pqcond}\frac1p+\frac1q>1,
\end{eqnarray}
then $f\in W_0(\mathbb R);$ and if $\frac1p+\frac1q<1$, then there
is a function $f$ such that $f\in L^p(\mathbb R)$, $f'\in
L^q(\mathbb R)$ and $f\not\in W_0(\mathbb R).$ Certainly, this result
essentially generalizes many previous results, e.g., the one in \cite{B}, but
we highlight it not only as a "landmark" but also since we shall pay
much attention to the case where $\frac1p+\frac1q=1$ and discuss why
it does not ensure the belonging to $W_0$, except a unique special
case $p=q=2$ (see the proof of Proposition~\ref{propWn2}). Despite the attraction of
(\ref{pqcond}), one could see already in \cite{Li} the
incompleteness of this condition. Indeed, it assumes not only the function to be
"good" but the derivative to be "good" as well. However,
related results on Fourier multipliers show that it is by no means
necessary. The model case is delivered by the well-known multiplier (see
\cite{Hir}, \cite[Ch.4, 7.4]{Stein}, \cite{Fef})
\begin{eqnarray}\label{must}
m(x)=m_{\alpha,\,\beta}(x)=\rho(x)\frac{e^{i|x|^\alpha}}
{|x|^\beta},
\end{eqnarray}
where $\rho$ is a $C^\infty$ function on $\mathbb R^n$, vanishing for $|x|\le 1$ and equal to $1$
if $|x|\ge 2$, with $\alpha, \beta>0$. Recall that the Fourier multipliers are defined as follows.
Let $m: \mathbb R^n\to \mathbb C$ be an almost everywhere bounded measurable function ($m\in
L_\infty(\mathbb R^n)$). Define on $L_2(\mathbb R^n)\cap L_p(\mathbb R^n)$
a linear operator $\Lambda$ by means of the following identity for
the Fourier transforms of functions $f\in L_2(\mathbb R^n)\cap L_p(\mathbb R^n)$:
\begin{eqnarray*} \mathscr{F}(\Lambda f)(x)=m(x)\mathscr{F}{f}(x).\end{eqnarray*}
If a constant $D>0$ exists such that for each $f\in L_2(\mathbb R^n)\cap L_p(\mathbb R^n)$ there holds
\begin{eqnarray*}
\|\Lambda f\|_{L_p(\mathbb R^n)}\le D\|f\|_{L_p(\mathbb R^n)},
\end{eqnarray*}
then $\Lambda$ is called a Fourier multiplier taking $L_p(\mathbb R^n)$ into
$L_p(\mathbb R^n).$ This is written as $m\in M_p$ and $\|m\|_{M_p}=\|\Lambda\|_{L_p\to L_p}.$
It is known (see~\cite{Mi_new}) that for $n\ge 1$ and $\a\neq 1$:
\begin{equation*}
{\rm if}
\qquad \frac{\beta}{\alpha}>\frac{n}{2},\qquad {\rm then}\quad m_{\alpha,\,\beta}\in
M_1 \quad ({\rm or}\quad m_{\alpha,\,\beta}\in M_\infty),
\end{equation*}
\begin{equation*}
{\rm if}\qquad \frac{\beta}{\alpha}\le\frac{n}{2},\qquad {\rm then}\quad m_{\alpha,\,\beta}\not\in M_1
\quad ({\rm or}\quad m_{\alpha,\,\beta}\not\in M_\infty).
\end{equation*}

\noindent The instance $\a\neq 1$ is also considered in~\cite{Mi_new}; more details for the case $\frac{\beta}{\alpha}=\frac{n}{2}$ can be found in \cite{Fef}, \cite{M}, and \cite{M2}.


To prove certain cases in these and related estimates, the following recent refinement of (\ref{pqcond}) (see \cite{K14})
is more convenient.

\begin{thmc}\label{CorKolIzv} Let $0< q\le\infty$, $1<r<\infty$, $s>\frac{n}r$, and
$f\in C_0(\R^n)$. Suppose that either $q=r=2$ or
\begin{equation}\label{eq.thmc}
\(1-\frac{n}{2s}\)\frac 1q+\(\frac{n}{2s}\)\frac 1r>\frac 12.
\end{equation}
If, in addition, $f\in L_{q}(\R^n)$ and $(-\D)^{\frac{s}2}f\in L_r(\R^n)$,
then $f\in W_0(\R^n)$.
\end{thmc}

This theorem is sharp in the following sense: {\it Let $1<
q,r<\infty$, $s\in\N$, and $s>\frac{n}r$. If we have
\begin{equation}\label{prop3.eq0}
\(1-\frac{n}{2s}\)\frac 1q+\(\frac{n}{2s}\)\frac 1r<\frac 12
\end{equation}
instead of (\ref{eq.thmc}), then there is a function $f\in C_0(\R^n)$ such that $f\in
L_{q}(\R^n)$ and $(-\D)^{\frac{s}2}f\in L_r(\R^n)$, but $f\not\in W_0(\R^n)$.}

It is worth mentioning that we study the case of equality in (\ref{prop3.eq0})
as well (see Proposition~\ref{propWn2} below).

Theorem~C, with $s\in \mathbb{N}$ and $\beta>([\frac{n}2]+1)(\alpha-1)$,
allows one to prove the well-known fact about $m_{\alpha,\,\beta}$
(cf. above): \emph{if $\beta>\frac{n\alpha}{2}$, then
$m_{\alpha,\beta}\in W_0(\mathbb{R}^n)$}. Note that the restriction $\beta>([\frac{n}2]+1)(\alpha-1)$ appears since the
derivatives of $m_{\alpha,\beta}$ may be "bad", that is, unbounded.

A natural way out of such a situation is to advert to weighted
spaces, in which the weight $w$ is chosen so that the growth of the
derivatives is dampened by means of the decay of the weight
function. Our goal is to obtain multidimensional analogues of (\ref{pqcond}), or
more precisely of Theorem~C in the case of weighted spaces.
New one-dimensional results follow as immediate consequences
of our general ones in the case $n=1$.

To give an idea of our main
results, we formulate here a representative one, a particular case of Corollary~\ref{CorpropI}. In what follows

\begin{equation}\label{eqWeighP}
  w_\e(x)=\(1+|x|^2\)^{\frac{\e}2},\quad \e\in \R.
\end{equation}

\begin{proposition}\label{propI}
Let $0< q< 2<r<\infty$, $s>0$, $\e>0$, and let $f\in
C_0(\R^n)$. If {\rm (\ref{eq.thmc})} holds,
\begin{equation*}
fw_{\frac{\e}{(1-\g)q}}\in L_{q}(\R^n),\quad{and}\quad (-\D)^{\frac{s}2}\(f
w_{-\frac{\e}{\g r} }\)\in L_{r}(\R^n),\quad    \g=\frac{2-q}{r-q},
\end{equation*}
then $f\in W_0(\R^n)$.
\end{proposition}

Further, one of the prominent and classical results on the interrelation between the function and its derivatives is the
Gagliardo-Nirenberg inequality, which is of the form (in fact, of one of the various possible forms, see~\cite{He}, \cite{HMOW}, or~\cite{Trib})
\begin{equation}\label{eqGN0}
    \Vert (-\D)^{\frac{\tau}2} f\Vert_{L_p}\le C\Vert f\Vert_{L_q}^{1-\d}\Vert   (-\D)^{\frac{s}2} f\Vert_{L_r}^\d,
\end{equation}
where $1<p,q,r<\infty$, $\tau, s\in \R$, $0\le \d\le 1$, and
\begin{equation*}
    \frac np-\tau=(1-\d)\frac nq+\d\(\frac nr-s\),\quad \tau\le \delta    s.
\end{equation*}
It turned out that for our aims it was beneficial to obtain weighted generalizations of (\ref{eqGN0}),
with natural and immediate applications to the absolute convergence of Fourier integrals. For example,
one such generalization, with weighted function $w_\e(x)$, $\e>0$, is as follows:
\begin{equation*}
    \Vert (-\D)^{\frac{\tau}2} f\Vert_{L_p}\le C\left\Vert f w_\e^{\frac1{q(1-\g)}}\right\Vert_{L_q}^{1-\d}\left\Vert
    (-\D)^{\frac{s}2}(f w_{-\e}^{\frac1{\g r}}) \right\Vert_{L_r}^\d,
\end{equation*}
where $\d$ is as above, but  $1<q<p<r<\infty$, $\tau, s\in \R$,
$0<\d< 1$, $s-\frac{n}r\neq -\frac{n}q$, $\e>0$, $\g=\frac{p-q}{r-q}$, and $\tau<\d s$.
However, since much is contained in \cite{MV}, we omit this part in the present work.

As for the prospects, it is worth mentioning that possible
directions of our interest are similar results in terms of the function
and its derivatives belonging to other spaces, more balance between the radial and non-radial cases,
and different methods for establishing the sharpness of the obtained results.

The outline of the paper is as follows. Since there are quite many different notations and technical tools,
we feel it necessary and convenient to present them in the next section. In Section~\ref{main}
we give our main results. Then Section~\ref{aux} is devoted
to auxiliary results. In Section~\ref{prf} we prove our main result Theorem~\ref{th1w}. Further, in Section~\ref{sharp}
we discuss various corollaries and establish the sharpness of the obtained results, more precisely, of the most
convenient of them Corollary~\ref{corWn1}.

\bigskip
\bigskip

\section{Basic notations and definitions}\label{not}

Let $\BR^n$ be the $n$-dimensional Euclidean space with elements
$x=(x_1,\dots,x_n)$, $\xi=(\xi_1,\dots,\xi_n)$ endowed with the scalar product
$(x,\xi)=x_1\xi_1+\dots+x_n\xi_n$ and norm $|x|=(x,x)^{\frac 12}$. As usual, the space $L_p(\BR^n,w)$
consists of measurable functions $f(x)$, $x\in\BR^n$, for which
$$
\Vert f\Vert_{L_p(w)}=\Vert f\Vert_{L_p(\R^n,w)}=\bigg(\int_{\BR^n}|f(x)|^pw(x)\,dx\bigg)^\frac1p<\infty,\quad 0<p<\infty,
$$
and
$$
\Vert f\Vert_{L_\infty(w)}=\Vert f\Vert_{L_\infty(\R^n,w)}=\,\mathrel{\mathop{\mbox{\rm
ess\,sup}}_{x\in\BR^n}}|f(x)|w(x)<\infty.
$$
If $w_0(x)\equiv 1$, we shall write $\Vert f\Vert_{p}=\Vert
f\Vert_{L_p(w_0)}$ and $L_p=L_p(w_0)$.
By $C(\R^n)$, we denote the space of all bounded uniformly continuous functions on $\R^n$.
We will also deal with the following class of functions
$$
C_0(\R^n)=\left\{f\,:\, f\in C(\R^n),\quad
\lim_{|x|\to\infty}f(x)=0\right\}.
$$

We will use standard notations for the space of tempered distributions
$\mathscr{S}'(\R^n)$ and for the corresponding space of test functions $\mathscr{S}(\R^n)$.
For $f\in L_1(\R^n)$, we denote its Fourier transform in a standard manner
$$
\widehat{f}(\xi)=\mathscr{F}f(\xi)=\int_{\R^n}f(x)e^{-i(\xi,x)}dx,
$$
and set also $\mathscr{F}^{-1}f(\xi)=\mathscr{F}f(-\xi)$. In the sequel, we shall
understand the operators $\mathscr{F}$ and $\mathscr{F}^{-1}$ in the distributional sense.

In this paper, the main theorems are formulated in terms of functions from the weighted
Besov spaces $B_{p,q}^s(\R^n,w)$. To define these spaces, let us consider a function
$\vp\in \mathscr{S}(\R^n)$ such that $\supp\vp\subset
\{\xi\in\R^n\,:\,1/2\le |\xi|\le 2\}$, $\vp(\xi)>0$ for $1/2< |\xi|<2$ and
\begin{equation*}\label{VV.eqRazbEd}
    \sum_{k=-\infty}^\infty \vp(2^{-k}\xi)=1\quad\text{if}\quad \xi\neq0.
\end{equation*}
We also introduce the functions $\vp_k$ and $\psi$ by means of the relations
$$
\mathscr{F}\vp_k(\xi)=\vp(2^{-k}\xi)
\quad
\text{and}
\quad
\mathscr{F}\psi(\xi)=1-\sum_{k=1}^\infty \vp(2^{-k}\xi).
$$

We begin with usual Besov spaces. Let $s\in\R$, $0<p,q\le\infty$. We will say that $f\in\mathscr{S}'(\R^n)$
belongs to the (non-homogeneous) weighted Besov space $B_{p,q}^s(\R^n,w)=B_{p,q}^s(w)$, if
$$
\Vert f\Vert_{B_{p,q}^s(w)}=\Vert
\psi*f\Vert_{L_p(w)}+\bigg(\sum_{k=1}^\infty 2^{sqk} \Vert \vp_k*
f\Vert_{L_p(w)}^q\bigg)^\frac 1q<\infty,
$$
with standard modification for  $q=\infty$.
If $w_0(x)\equiv 1$, then $B_{p,q}^s(\R^n,w_0)=B_{p,q}^s$ are non-weighted Besov spaces.

Now, in order to define homogeneous Besov spaces, recall that

$$
\dot{\mathscr{S}}(\R^n)=\left\{\vp\in{\mathscr{S}}(\R^n)\,:\, (D^\nu
\widehat{\vp})(0)=0\,\,\,\text{for\ all}\,\,\,\nu\in\N^n\cup\{0\}\right\},
$$
where $\dot{\mathscr{S}}'(\R^n)$ is the space of all continuous functionals on $\dot{\mathscr{S}}(\R^n)$.
We will say that $f\in \dot{\mathscr{S}}'(\R^n)$ belongs to the homogeneous Besov space $\dot B_{p,q}^s$, if
$$
\Vert f\Vert_{\dot B_{p,q}^s}=\bigg(\sum_{k=-\infty}^\infty 2^{sqk}
\Vert \vp_k* f\Vert_p^q\bigg)^\frac 1q<\infty,
$$
with standard modification for $q=\infty$.

We shall also deal with weighted spaces of Bessel potentials $H_p^s(\R^n,w)=H_p^s(w)$. Let $1<p<\infty$ and $s\in\R$.
We will say that $f\in {\mathscr{S}}'(\R^n)$  belongs to the space $H_p^s(w)$  if
$$
\Vert f\Vert_{H_p^s(w)}=\Vert (I-\D)^\frac s2 f\Vert_{L_p(w)}<\infty.
$$

Now a discussion on weighted functions is in order. We shall deal with the so-called admissible
weights.

\begin{definition}\label{admwei} We will say that a measurable function $w:\R^n\mapsto
\R_+$ is an \emph{admissible weight}, written $w\in \mathcal{W}^n$, if

1) $w\in C^\infty(\R^n)$;

2) For each multi-index $\g$, there is a positive constant $c_\g$ such that
$$
\left|\frac{\partial^{\g_1+\dots \g_n}}{\partial x_1^{\g_1}\dots\partial x_n^{\g_n}} w(x)\right|\le c_\g w(x)\quad\text{for each}\quad x\in\R^n;
$$

3) There are constants $c>0$ and $\a\ge 0$ such that
$$
0<w(x)\le cw(y)(1+|x-y|^2)^{\frac{\a}2}\quad\text{for\ any}\quad x,y\in\R^n.
$$
\end{definition}
Note that if $w\in \mathcal{W}^n$ and $w'\in \mathcal{W}^n$, then $w^{-1}\in \mathcal{W}^n$ and $ww'\in \mathcal{W}^n$.
The functions
$$
w(x)=(1+|x|^2)^{\frac{\a}2},\quad v(x)=(1+\log(1+|x|^2))^\a,\quad \a\in\R,
$$
are typical examples of admissible weights.

In the sequel, we shall denote by $C$ (or $C$ with indicated parameters) absolute positive constants
(or constants depending only on the indicated parameters, respectively), while by $A$ certain specific
finite constants. We also set $\frac{\infty}{\infty}=1$ and $\frac{0}{0}=0$.

\bigskip

\section{Statement of the main results}\label{main}

The first result that we present cannot be called our main one only
because it does not lead explicitly to the belonging of the considered
function to Wiener's algebra. It is an important tool for us rather
than one of the intended goals. However, these Gagliardo-Nirenberg type
inequalities for homogeneous Besov spaces are of similar nature and
are of interest in their own right. Also, our main result that follows immediately
will be more transparent.

Let in what follows
$$
\g(p,q,r)=\left\{
         \begin{array}{ll}
           \displaystyle\frac{p-q}{r-q}, & \hbox{$q<p$;} \\
           \\
           \displaystyle\frac{q-p}{q-r}, & \hbox{$r<p$.}
         \end{array}
       \right.
$$

\begin{theorem}\label{th1GN}
Let either $0< q< p<r\le\infty$ or $0< r< p<q\le\infty$, and let
$\tau,\s,s\in\R$, $\s<s$, $0<\eta<\infty$. Suppose that $u$, $v$
are measurable functions on $\R^n$ such that
\begin{equation}\label{th1GN.eq0}
1\le u(x)v(x)\quad\text{for\ all}\quad x\in\R^n,
\end{equation}
and also
\begin{equation}\label{th1GN.eq1}
 \s-\tau<n\(\frac1q-\frac1p\)\quad \text{and}\quad   1\le u(x),\quad\text{if}\quad
 0<q<p
\end{equation}
and
\begin{equation}\label{th1GN.eq2}
  n\(\frac1r-\frac1p\)<s-\tau\quad \text{and}\quad  1\le v(x),\quad\text{if}\quad 0<r<p.
\end{equation}
If
\begin{equation}\label{th1GN.eq3}
\frac {1-\t}q+\frac \t r>\frac 1p\,,\quad \t=\frac{\tau-\s}{s-\s}\,,
\end{equation}
and
\begin{equation*}
f\in \dot{B}_{q,\infty}^\s\(u^{1/{(1-\g)}}\)\cap
\dot{B}_{r,\infty}^s\(v^{1/\g}\),\quad
    \g=\g(p,q,r),
\end{equation*}
then $f\in \dot{B}_{p,\eta}^\tau$. If, in addition, $s-\frac{n}r\neq
\s-\frac{n}q$, then
\begin{equation*}
    \Vert f\Vert_{\dot{B}_{p,\eta}^\tau}\le C\Vert f\Vert_{\dot{B}_{q,\infty}^\s\(u^{1/{(1-\g)}}\)}^{1-\d}\Vert
    f\Vert_{\dot{B}_{r,\infty}^s\(v^{1/\g}\)}^\d,
\end{equation*}
where $\d=\frac{\tau-\s+n(\frac1q-\frac1p)}{s-\s+n(\frac1q-\frac1r)}$ and $C$ is a
constant independent of $f$.
\end{theorem}

The next theorem is our main result.
\begin{theorem}\label{th1w}
Let either $0< q< 2<r\le\infty$ or $0< r< 2<q\le\infty$, $\s<s$, and
let $f\in C_0(\R^n)$. Suppose that $u, v$ are measurable
functions on $\R^n$ such that {\rm (\ref{th1GN.eq0})} holds, while
{\rm (\ref{th1GN.eq1})}, {\rm (\ref{th1GN.eq2})}, and {\rm
(\ref{th1GN.eq3})} are valid for $\tau=\frac{n}2$ and $p=2$.

If

\begin{equation*}
f\in \dot{B}_{q,\infty}^\s\(u^{{1/(1-\g)}}\)\cap
\dot{B}_{r,\infty}^s\(v^{1/\g}\),\quad
    \g=\g(2,q,r),
\end{equation*}
then $f\in W_0(\R^n)$. If, in addition, $s-\frac{n}r\neq \s-\frac{n}q$, then
\begin{equation*}
    \Vert f\Vert_{W_0}\le C\Vert f\Vert_{\dot{B}_{q,\infty}^\s\(u^{1/(1-\g)}\)}^{1-\d}\Vert
    f\Vert_{\dot{B}_{r,\infty}^s\(v^{{1/\g}}\)}^\d,
\end{equation*}
where $\d=\frac{\frac{n}q-\s}{s-\s+n(\frac1q-\frac1r)}$ and $C$ is a constant
independent of $f$.
\end{theorem}

As for the consequences, let us start with those for the weights from $\mathcal{W}^n$. The following statement is less general but presents more effective sufficient conditions. By this we pass from homogeneous spaces, we have worked so far, to non-homogeneous Sobolev type spaces. This allows one to apply the well known embedding theorems.

\begin{corollary}\label{corWn1}
Let either $0< q< 2<r<\infty$ or $1< r< 2<q\le \infty$, $s>0$, and
let $f\in C_0(\R^n)$. Suppose that $u, v\in \mathcal{W}^n$ are such that
{\rm (\ref{th1GN.eq0})} holds, while {\rm (\ref{th1GN.eq1})} and
{\rm (\ref{th1GN.eq2})} are valid for $\tau=\frac{n}2$, $\s=0$, and
$p=2$. If

\begin{equation}\label{corWn1.eq3}
\bigg(1-\frac{n}{2s}\bigg)\frac 1q+\bigg(\frac{n}{2s}\bigg)\frac 1
r>\frac 12
\end{equation}
and
\begin{equation*}
f\in L_{q}\(u^{{1/(1-\g)}}\)\cap H_{r}^s\(v^{1/\g}\),\quad
    \g=\g(2,q,r),
\end{equation*}
then $f\in W_0(\R^n)$. If, in addition, $s-\frac{n}r\neq
-\frac{n}q$, then
\begin{equation*}
    \Vert f\Vert_{W_0}\le C\Vert f\Vert_{L_q\(u^{{1/(1-\g)}}\)}^{1-\d}\Vert
    f\Vert_{H_r^s\(v^{{1/\g}}\)}^\d,
\end{equation*}
where $\d=\frac{\frac{n}q}{s+n(\frac1q-\frac1r)}$ and $C$ is a constant independent of $f$.
\end{corollary}

Now, we return, in a sense, to homogeneity of the spaces. The following corollary is a direct (weighted) extension of~\eqref{pqcond}.


\begin{corollary}\label{CorpropI}
Let  $0<q<2<r<\infty$, $s>0$, and
let $f\in C_0(\R^n)$. Suppose that $u, v\in \mathcal{W}^n$ are such that
$1\le u(x)v(x)$ and $1\le u(x)$. If~\eqref{corWn1.eq3} holds, and
\begin{equation*}
fu^{{1/q(1-\g)}}\in L_{q}(\R^n)\quad{and}\quad (-\D)^{\frac{s}2}\(f
v^{{1/r\g}}\)\in L_{r}(\R^n),\quad    \g=\frac{2-q}{r-q},
\end{equation*}
then $f\in W_0(\R^n)$.
\end{corollary}

Recall that if $f(x)=f_0(\rho)$, $\rho=|x|$, then using the well-known equality
$$
\Delta f(x)=\frac{\partial^2 f_0(\rho)}{\partial \rho^2}+\frac{n-1}{\rho}\frac{\partial f_0(\rho)}{\partial \rho},
$$
we get the following version of Corollary~\ref{CorpropI} for radial functions.

\begin{corollary}\label{lemrad}
Let  $0<q<2<r<\infty$ and
let $f\in C_0(\R^n)$. Suppose that $u, v\in \mathcal{W}^n$ are such that
$1\le u(x)v(x)$ and $1\le u(x)$, and, in addition, let $s$ be even and $f(x)=f_0(\rho)$, $u(x)=u_0(\rho)$, and $v(x)=v_0(\rho)$. If
\eqref{corWn1.eq3} holds, and
\begin{equation*}
f_0\in L_{q}\(\rho^{n-1}u_0^{{1/(1-\g)}}(\rho)\),\quad \frac1{\rho^{s-\nu}}\frac{\partial^\nu f_0}{\partial \rho^\nu}\in L_r\(\rho^{n-1}v_0^{1/\g}(\rho)\),\quad
    \g=\frac{2-q}{r-q},\quad \nu=1,\dots,s,
\end{equation*}
then $f\in W_0(\R^n)$.
\end{corollary}

\begin{rem}
It is worth noting that Corollaries~\ref{corWn1}--\ref{lemrad}
hold true for
a wider class of weights. In particular, for the class $\mathcal{W}_e^n$
(see~\cite{ScII}), which differs from $\mathcal{W}^n$ in the way that it suffices to claim
in its definition the weaker condition
$$
0<w(x)\le C w(x-y)e^{d|y|},\quad x,y\in\R^n,
$$
where $C>0$ and $d>0$ are some fixed constants, in place of 3) in Definition \ref{admwei}.
For example, $\mathcal{W}^n_e$ contains functions like
$$
w(x)=e^{\pm |x|^\b},\quad 0<\b\le 1.
$$
\end{rem}

\bigskip

\section{Auxiliary results}\label{aux}

If the weight is admissible, the following lemma, which relates
usual and weighted Besov spaces, is of importance (see~\cite[p.156]{ET}).

\begin{lemma}\label{lemVesBes}
Let $0<p,q\le\infty$, $s\in\R$, and $w\in \mathcal{W}^n$. Then the operator
$f\mapsto w^{1/{p^*}} f$ ($p^*=p$, if $p<\infty$, and $p^*=1$ for
$p=\infty$) is an isomorphism $B_{p,q}^s \(w^{1/p}\)$ on
$B_{p,q}^s$. In particular, $\Vert f w^{1/{p^*}}\Vert_{B_{p,q}^s}$
is an equivalent quasinorm on $B_{p,q}^s(w)$.
\end{lemma}

A similar lemma is valid for the spaces $H_p^s(\R^n, w)$, $w\in \mathcal{W}^n$.

\medskip

We shall make use of the embeddings for the Besov spaces, by presenting
the next result which states that the spaces $B_{p,q}^k$, $W_p^k$, and $C^k$  are related in the following well-known way.
\begin{lemma}\label{lemVlozh2}
Let $s>0$, $k\in\N\cup\{0\}$, and $1< p<\infty$. Then
$$
B_{p,1}^s\subset H_p^s\subset B_{p,\infty}^s
$$
and if $1\le p<\infty$, then
$$
B_{p,1}^k\subset W_p^k\subset B_{p,\infty}^k.
$$
\end{lemma}

Here and in what follows, $\subset$ means a continuous embedding.
The proof of Lemma~\ref{lemVlozh2} can be found, e.g.,
in~\cite[Ch.2 and Ch.5]{TribF} and \cite[Ch.3 and Ch.11]{Pe}.
It readily follows from Lemma~\ref{lemVesBes} that the above-mentioned embeddings
are also valid for the weighted Besov spaces $B_{p,q}^s(w)$ provided, of course, that $w\in \mathcal{W}^n$.

\medskip

The following lemma is a simple corollary of H\"older's inequality.

\begin{lemma}\label{lem2}
Let $0<q<p<r\le\infty$, $\g=\frac{p-q}{r-q}$, and $u, v$ be
positive measurable functions on $\R^n$ such that $u(x)v(x)\ge 1,$ $x\in\R^n.$ Then
$$
\Vert f\Vert_p\le \Vert  f\Vert_{L_q\(u^{{1/(1-\g)}}\)}^{\frac{q(1-\g)}2}\Vert
f\Vert_{L_r\(v^{{1/\g}}\)}^{\frac{r\g}2}.
$$
\end{lemma}

\bigskip
\bigskip

\section{Proof of the main results}\label{prf}

\begin{proof}[Proof of Theorem~\ref{th1GN}]

In order to prove that $f\in\dot{B}_{p,\eta}^\tau$, it suffices to check that
\begin{equation}\label{pr4.eq1}
\sum_{k=-\infty}^\infty \(2^{\tau k}\Vert \vp_k*f\Vert_{L_p}\)^\eta
=\sum_{k=0}^\infty\,\,\,+\,\,\,\sum_{k=-\infty}^{-1}=S_1+S_2<\infty.
\end{equation}

Let first $0<q<p<r\le\infty$.

Using the embedding $\dot{B}_{q,\infty}^\s\subset \dot
B_{p,\infty}^{\s_1}$, where $\s_1=\s-n(\frac1q-\frac1p)$ (see, for
example,~\cite[6.5.1]{BL}), we obtain
\begin{equation}\label{pr4.eq3}
\begin{split}
    S_2&=\sum_{k=1}^\infty \(2^{-\tau k}\Vert \vp_{-k}*f\Vert_{L_p}\)^\eta\\
    &\le \sup_{j\in \Z_+}\(2^{-\s_1 j}\Vert
\vp_{-j}* f\Vert_{L_p}\)^\eta\sum_{k=1}^\infty 2^{-k(\tau-\s_1)\eta}\\
&\le C \Vert f\Vert_{\dot{B}_{p,\infty}^{\s_1}}^\eta\le C \Vert
f\Vert_{\dot{B}_{q,\infty}^{\s}}^\eta\le  C \Vert
f\Vert_{\dot{B}_{q,\infty}^{\s}(u^{{1/(1-\g)}})}^\eta.
\end{split}
\end{equation}

Let us proceed to the sum $S_1$. Denoting $\l=\frac{r(p-q)}{p(r-q)}$ and
applying Lemma~\ref{lem2}, we get
\begin{equation}\label{pr4.eq4}
    \begin{split}
 S_1&\le \sum_{k=0}^\infty \(2^{\tau k} \Vert
 \vp_k*f\Vert_{L_q(u^{{1/(1-\g)}})}^{1-\l}\Vert \vp_k*f\Vert_{L_r(v^{{1/\g}})}^{\l}\)^\eta\\
 &\le\bigg(\sup_{j\in\Z_+} 2^{\s j}\Vert
 \vp_j*f\Vert_{L_q(u^{{1/(1-\g)}})}\bigg)^{\eta(1-\l)}\\
&\times\bigg(\sup_{j\in\Z_+} 2^{s j}\Vert
 \vp_j*f\Vert_{L_r(v^{{1/\g}})}\bigg)^{\eta\l}\sum_{k=0}^\infty 2^{(\tau-(1-\l)\s-\l s)\eta k}\\
 &=\Vert f\Vert_{\dot{B}_{q,\infty}^\s(u^{{1/(1-\g)}})}^{(1-\l)\eta}\Vert f\Vert_{
 \dot{B}_{r,\infty}^s(v^{{1/\g}})}^{\l \eta}\sum_{k=0}^\infty
 2^{\frac{qr(s-\s)\eta}{r-q}(\frac1p-(1-\t)\frac1q-\t\frac1r)k}\\
 &\le C\Vert f\Vert_{\dot{B}_{q,\infty}^\s(u^{{1/(1-\g)}})}^{(1-\l)\eta}\Vert f\Vert_{
 \dot{B}_{r,\infty}^s(v^{{1/\g}})}^{\l \eta}.
    \end{split}
\end{equation}
Hence, combining (\ref{pr4.eq3}) and (\ref{pr4.eq4}), we derive from (\ref{pr4.eq1}) that
\begin{equation*}\label{pr4.eq4.5}
\Vert f\Vert_{\dot B_{p,\eta}^\tau}\le C\left\{\Vert
f\Vert_{\dot{B}_{q,\infty}^\s(u^{{1/(1-\g)}})}+\Vert
f\Vert_{\dot{B}_{q,\infty}^\s (u^{{1/(1-\g)}})}^{1-\l}\Vert
f\Vert_{\dot{B}_{r,\infty}^s(v^{{1/\g}})}^{\l}\right\}.
\end{equation*}
It remains to pass to the product inequality, that is, to substitute
$f\to f(\e \cdot)$ and $u\to u(\e \cdot)$, $v\to v(\e \cdot)$, where
$$
\e=\(\frac{\Vert f\Vert_{\dot{B}_{q,\infty}^\s(u^{{1/(1-\g)}})}}{
\Vert f\Vert_{\dot{B}_{r,\infty}^s(v^{{1/\g}})}
}\)^\frac1{s-\s+n(\frac1q-\frac1r)}
$$
and to
use the property of homogeneity for Fourier transform (see, for
example,~\cite[3.4.1]{TribF} and~\cite[7.2]{N}).

The case $0<r<p<q\le\infty$ is considered in a similar manner.

Theorem~\ref{th1GN} is proved.  \end{proof}

\begin{proof}[Proof of Theorem~\ref{th1w}]
It is well known (see, e.g.,~\cite{Bes},~\cite{P} or \cite[p.119]{Pe}) that if $f\in C_0\cap \dot B_{2,1}^\frac n2$, then $f\in  W_0(\R^n)$ and
$
    \Vert f\Vert_{W_0}\le C\Vert f\Vert_{\dot B_{2,1}^{\frac{n}2}},
$
where $C$ is a constant independent of $f$. Hence, the theorem
obviously follows from Theorem~\ref{th1GN}.
\end{proof}

\begin{proof}[Proof of Corollary~\ref{corWn1}]
This corollary can be easily obtained by using
Lemma~\ref{lemVesBes}, Lemma~\ref{lemVlozh2}, and Theorem~\ref{th1w} with $\s=0$.
\end{proof}


\section{Certain consequences and sharpness}\label{sharp}

We will show both applicability and sharpness of the obtained above results
by making use of the familiar power weight function $w_\e(x)=\(1+|x|^2\)^{\frac{\e}2}$.

We first show that Corollary~\ref{corWn1} can be applied to the
study of the function $m_{\a,\b}$ in the case where its derivatives are unbounded.

\begin{corollary}\label{corWm} Let $\a>0,\b>0$, and $\a\neq 1$. If
$\b>\frac{n\a}2$, then $m_{\a,\b}\in W_0(\mathbb R^n).$
\end{corollary}

\begin{proof}
It suffices to consider only the case $\b>n$. Indeed, if $\b\le n$,
then Corollary~\ref{corWm} follows from the non-weighted sufficient
condition (see Corollary~\ref{corWn1}).

Observe first that condition $2\b>n\a$ allows one to choose $\e>0$ so that
$$
\b>n+\e\quad\text{and}\quad \b-n(\a-1)>-\e.
$$
We then are able to choose $q$, close enough to $1$, and large enough $r$ such that
\begin{equation}\label{corWm.eq1}
    \b-\frac{\e(r-q)}{q(r-2)}>\frac nq\quad\text{and}\quad
\b-n(\a-1)+\frac{\e(r-q)}{r(2-q)}>\frac nr,
\end{equation}
which, in turn, means that
$$
m_{\a,\b}\in L_q\(w_{\frac{\e}{1-\g}}\)\cap
H_r^n\(w_{-\frac{\e}{\g}}\),\quad \g=\frac{2-q}{r-q}.
$$
It follows from inequalities (\ref{corWm.eq1}) that
\begin{equation*}
    \bigg(\frac1q+\frac1r-1\bigg)\bigg(n+\frac{2\e(r-q)}{(2-q)(r-2)}\bigg)<2\b-n\a.
\end{equation*}
This allows us to choose, in addition, the parameters $q$ and $r$ so that
$$\frac1q+\frac1r>1.$$
Therefore, by Corollary~\ref{corWn1}, we obtain $m_{\a,\b}\in W_0(\R^n)$.
\end{proof}

The following notions and auxiliary statements are needed for
proving the sharpness of the obtained results.

Further, let us denote
$$
\mu_{\a,\b}(x)=\frac{m_{\a,\b}(x)}{\log |x|}
$$
and
$$
\nu_\a(x)=\mu_{\a,\frac{n\a}2}(x).
$$

\begin{lemma}\label{lem1}
Let $\b>0$.

(i) If  $0\le \a<1$, then
$$
|\mathscr{F}\mu_{\a,\b}(\xi)|\sim {|\xi|^{-\frac{n-\b-\frac{\a n}2}{1-\a}}}(\log
\frac1{|\xi|})^{-1},\quad |\xi|\to 0.
$$

(ii) If $\a>1$, then
$$
|\mathscr{F}\mu_{\a,\b}(\xi)|\sim {|\xi|^{-\frac{n-\b-\a n/2}{1-\a}}}(\log
|\xi|)^{-1},\quad |\xi|\to \infty.
$$
\end{lemma}

In the case $0<\b+\frac{n}2<\frac12$, see the proof of this lemma in~\cite[Lemma~4]{M}. The
general case can be proved by using (modifying) the proofs of Lemmas 2.6 and 2.7 in~\cite{M2}.

We now show that Corollary~\ref{corWn1} is sharp. First, the next proposition holds true.

\begin{proposition}\label{propWn2}
Let either $1\le q< 2<r< \infty$ or $1\le r< 2<q< \infty$, $s>\frac{n}2$, $s\in\N$. If
\begin{equation}\label{propWn2.eq1}
\bigg(1-\frac{n}{2s}\bigg)\frac 1q+\bigg(\frac{n}{2s}\bigg)\frac 1r\le \frac 12,
\end{equation}
then for each $\e>0$ there is a function $f\in C_0(\R^n)$ such that either
\begin{equation*}
f\in L_q\(w_{\frac{\e}{1-\g}}\)\cap
H_r^s\(w_{-\frac{\e}{\g}}\),\quad
    \g=           \frac{2-q}{r-q}
\end{equation*}
or
\begin{equation*}
f\in L_q\(w_{-\frac{\e}{1-\g}}\)\cap
H_r^s\(w_{\frac{\e}{\g}}\),\quad
    \g=           \frac{q-2}{q-r},
\end{equation*}
but $f\not\in W_0(\R^n)$.
\end{proposition}

\begin{proof}
Consider first the instance $1\le q< 2<r< \infty$. Let $\a,\b>0$, $\a\neq 1$ be such that
\begin{equation}\label{propWn2.eq2}
    \b-\frac\e{1-\g}\ge \frac nq,\quad \b-s(\a-1)+\frac\e\g\ge \frac nr.
\end{equation}
It follows from inequality (\ref{propWn2.eq1}) that
\begin{equation}\label{propWn2.eq3}
    2\bigg(\bigg(1-\frac{n}{2s}\bigg)\frac 1q+\bigg(\frac{n}{2s}\bigg)\frac 1
r-\frac 12\bigg)\bigg(n+\frac{2\e(r-q)}{(2-q)(r-2)}\bigg)\le2\b-n\a.
\end{equation}
If the inequalities in (\ref{propWn2.eq1})--(\ref{propWn2.eq3}) are strict, then for proving
the proposition it suffices to consider the function $f=m_{\a,\b}$. In particular,
(\ref{propWn2.eq3}) implies that $\a>0$, $\a\neq 1$, and $\b>0$ can be chosen in such a way
that $2\b-n\a\le 0$, which leads to $f\not\in W_0(\R^n)$.

If there are equalities in some of the (\ref{propWn2.eq1})--(\ref{propWn2.eq3}), we take
$f(x)=\nu_{\frac2q}(x)$.
It is easy to check that $f\in  L_q\(w_{\frac{\e}{1-\g}}\)\cap
H_r^s\(w_{-\frac{\e}{\g}}\)$ in this case. Further, since $1\le q<2$,
Lemma~\ref{lem1} (ii) implies
$$
|\mathscr{F}f(\xi)|\sim \({|\xi|^n\log |\xi|}\)^{-1},\quad
|\xi|\to\infty,
$$
which gives $f\not \in W_0(\R^n)$.

The proof is quite similar for $q>2$. We just mention that in that case (i) of Lemma~\ref{lem1}
should be used, which gives for $f$
$$
|\mathscr{F}f(\xi)|\sim \({|\xi|^n\log \frac1{|\xi|}}\)^{-1},\quad
|\xi|\to 0.
$$
The proof is complete.
\end{proof}

We have shown that condition (\ref{corWn1.eq3}) on the parameters
$q$ and $r$ in Corollary~\ref{corWn1} is sharp. Let us show that
conditions (\ref{th1GN.eq1}) and (\ref{th1GN.eq2}) in Corollary~\ref{corWn1}
are sharp as well. We restrict ourselves to condition (\ref{th1GN.eq1}).

\begin{proposition}\label{propWn4}
Let $1\le q< 2<r< \infty$, $s>\frac n2$, $s\in\N$. If
\begin{equation}\label{propWn4.eq1}
\bigg(1-\frac{n}{2s}\bigg)\frac 1q+\bigg(\frac{n}{2s}\bigg)\frac 1r> \frac 12,
\end{equation}
then for each $\e>\frac{n(r-2)(2-q)}{2(r-q)}$ there is a function $f\in C_0(\R^n)$ such that
\begin{equation*}
f\in L_q\(w_{-\frac{\e}{1-\g}}\)\cap
H_r^s\(w_{\frac{\e}{\g}}\),\quad  \g=\frac{2-q}{r-q},
\end{equation*}
but $f\not\in W_0(\R^n)$.
\end{proposition}

\begin{proof} One can use the function $f(x)=m_{\a,\b}(x)$ as in the proof of Proposition~\ref{propWn2}.
Let $\a,\b>0$, $\a\neq 1$, be such that
\begin{equation*}
    \b-\frac\e{1-\g}> \frac nq,\quad \b-s(\a-1)+\frac\e\g> \frac nr.
\end{equation*}
Then it follows from inequality (\ref{propWn4.eq1}) that
\begin{equation*}
    2\bigg(\bigg(1-\frac{n}{2s}\bigg)\frac 1q+\bigg(\frac{n}{2s}\bigg)\frac 1
r-\frac 12\bigg)\bigg(n-\frac{2\e(r-q)}{(2-q)(r-2)}\bigg)< 2\b-n\a.
\end{equation*}
It remains to make use of the fact that $f\not\in
W_0(\R^n)$ if $2\b\le n\a$ and choose appropriate $\a$ and $\b$.
\end{proof}

It is worth mentioning that Corollary~\ref{corWn1} is, generally
speaking, invalid for $q=r=2$ and non-trivial weights, contrary to
Beurling's theorem in \cite{B}, which is valid for $q=r=2$ and $u=v\equiv 1$.

\begin{proposition}\label{propWn4}
Let $\e>-\frac n2$ and $\e\neq 0$. Then there is a function $f\in
C_0(\R^n)$ such that
\begin{equation*}
f\in L_{2}(w_{\e})\cap H_{2}^n(w_{-\e}),
\end{equation*}
but $f\not\in W_0(\R^n)$.
\end{proposition}

\begin{proof} It suffices to consider the function $f(x)=\nu_{1+\frac{2\e}n}(x),$
and apply Lemma~\ref{lem1}.
\end{proof}

\bigskip

\section*{Acknowledgements}
Both authors were the students of R.M. Trigub at different times.
It is our pleasure to thank him for years of support, numerous discussions
and for the store house we obtained from him. In addition, it is worth mentioning that the entire subject matter originated from Trigub's conjecture (\ref{pqcond}).

The authors also thank A. Miyachi, S. Samko, and I. Verbitsky for valuable discussions.

This research has received funding from the European Union's Horizon 2020 research and innovation
programme under the Marie Sklodowska-Curie grant agreement No~704030.
The first author also acknowledges the support of the Gelbart Institute at
the Department of Mathematics of Bar-Ilan University.


\bigskip
\bigskip


\begin{thebibliography}{99}


\bibitem{BL} J. Bergh and J.~L\"ofsr\"om, \emph{Interpolation spaces. An introduction}, Springer, 1976.

\bibitem{Bes}
O.V. Besov, {\it H\"ormander's theorem on Fourier multipliers},
Trudy Mat. Inst. Steklov {\bf 173} (1986), 164--180 (Russian). -
English transl. in Proc. Steklov Inst. Math., {\bf 4} (1987), 4--14.

\bibitem{B} {A. Beurling}, \emph{Sur les integrales de Fourier absolument convergentes
et leur application \`a fonctionell}, {Proc. IX Congr\`es de Math. Scand.} Helsingfors (1938),  345--366.

\bibitem{ET} {D.E. Edmunds and H. Triebel}, {\it Function spaces, entropy numbers,
differential operators},  Cambridge Univ. Press, Cambridge, 1996.

\bibitem{Fef}
Ch. Fefferman, {\it Inequalities for Strongly Singular Convolution
Operators}, Acta Math. {\bf 124} (1970), 9--36.

\bibitem{HMOW} H. Hajaiej, L. Molinet, T. Ozawa, and B. Wang,
{\it Necessary and sufficient conditions for the fractional
Gagliardo-Nirenberg inequalities and applications to Navier-Stokes
and generalized boson equations. Harmonic analysis and nonlinear
partial differential equations}, 159--175, RIMS K\^oky\^uroku
Bessatsu, B26, Res.Inst. Math. Sci. (RIMS), Kyoto, 2011.

\bibitem{He} L.I. Hedberg,  \emph{On certain convolution inequalities}, Proc. Amer. Math. Soc. \textbf{36} (1972), 505--510.

\bibitem{Hir}
I.I. Hirschman, Jr., {\it On multiplier transformations, I}, Duke
Math. J. {\bf 26} (1959), 221--242; {\it II}, ibid {\bf 28} (1961), 45--56.

\bibitem{K14} Yu.S. Kolomoitsev, {\it Multiplicative sufficient conditions for Fourier multipliers}, Izv. RAN, Ser.Mat.
{\bf 78:2} (2014), 145--166 (Russian). - English transl. in Math. Russian Izv. {\bf 78} (2014), 354--374.

\bibitem{KL} Yu. Kolomoitsev and E. Liflyand, {\it Absolute convergence of multiple Fourier integrals},
Studia Math. {\bf 214} (2013), 17--35.

\bibitem{Li}
E. Liflyand, {\it On absolute convergence of Fourier integrals}, Real Anal. Exchange {\bf 36} (2010/11), 353--360.

\bibitem{LiOu}
E. Liflyand and E. Ournycheva, {\it Two spaces conditions for integrability of the Fourier
transform}, Analysis, {\bf 28} (2008), 429--443.

\bibitem{LST} E. Liflyand, S. Samko, and R. Trigub, {\it The Wiener algebra of absolutely convergent Fourier integrals:
an overview}, Anal. Math. Phys. {\bf 2} (2012), 1--68.

\bibitem{LiTr}
E. Liflyand and R. Trigub, {\it On the Representation of a Function as an Absolutely Convergent Fourier Integral},
Trudy Mat. Inst. Steklov {\bf 269} (2010),  153--166 (Russian). - English transl.:
Proc. Steklov Inst. Math. {\bf 269} (2010), 146--159.

\bibitem{LiTr1}
E. Liflyand and R. Trigub, {\it Conditions for the absolute
convergence of Fourier integrals}, J. Approx. Theory {\bf 163} (2011), 438--459.

\bibitem{Lofstrom} J.~L\"ofsr\"om, {\it Besov spaces in theory of approximation},
{Ann. Mat. Pura Appl.} {\bf 85} (1970), 93--184.

\bibitem{MV} M. Meyries and M. Veraar, {\it  Traces and embeddings of anisotropic function spaces},
Math. Ann. {\bf 360} (2014), 571--606.


\bibitem{M}  {A. Miyachi}, {\it On some Fourier multipliers for
$H^p(\R^n)$}, J. Fac. Sci. Univ. Tokyo Sect. IA Math. \textbf{27} (1980), 157--179.

\bibitem{Mi_new} {A. Miyachi}, {\it On some singular Fourier multipliers}, J. Fac. Sci. Univ. Tokyo Sect. IA Math. \textbf{28} (1981), 267--315.


\bibitem{M2}  {A. Miyachi}, \emph{Notes on Fourier multipliers for $H_p$, $BMO$ and the Lipschitz spaces},
J. Fac. Sci. Univ. Tokyo Sect. IA Math. \textbf{30} (1983), no. 2, 221--242.

\bibitem{N} S.M. Nikol'skii,
{\it The Approximation of Functions of Several Variables and the Imbedding Theorems}, 2nd ed.
{Moscow}: Nauka, 1977 (Russian). - English transl. of 1st. ed.: John Wiley $\&$ Sons, New-York, 1978.

\bibitem{P} {J. Peetre}, \emph{Applications de la th\'eorie des espaces d'interpolation
dans l'analyse harmonique}, {Ricerche Mat.} \textbf{15} (1966), 3--36.

\bibitem{Pe} {J. Peetre}, \emph{New Thoughts on Besov Spaces},  Durham, NC: Duke University Press, 1976.

\bibitem{Samko}
S.G. Samko, {\it The spaces {${L}_{p, r}^{\alpha}({R}^n)$} and
hypersingular integrals}, Studia Math. {\bf 61} (1977), 193--230 (Russian).

\bibitem{ScII} Th. Schott, \emph{Function spaces with exponential weights II}, Math. Nachr. 196 (1998), 231--250.

\bibitem{Stein}
E.M. Stein, {\it Singular Integrals and Differentiability Properties of Functions}, Princeton Univ. Press, Princeton, N.J., 1970.

\bibitem{SW} E. Stein and G. Weiss, {\it Introduction to Fourier
Analysis on Euclidean Spaces}, Princeton Univ. Press, Princeton, N.J., 1971.

\bibitem{TribF} H. Triebel, \emph{Theory of function spaces}, Monographs in
Mathematics, vol. 78, Birkh\"auser Verlag, Basel, 1983.

\bibitem{Trib} H. Triebel, {\it Gagliardo-Nirenberg inequalities}, Proc. Steklov Inst. Math. {\bf 284}  (2014), 263--279.

\bibitem{TrUMZ}  R.M. Trigub, {\it On Fourier multipliers and absolute convergence of Fourier integrals
of radial functions}, Ukr. Mat. Zh. {\bf 62} (2010), 1280--1293 (Russian). - English transl. in
Ukrainian Math. J. {\bf 62} (2011), 1487--1501.

\bibitem{TB} R.M. Trigub and E.S. Belinsky, {\it Fourier Analysis and Appoximation of Functions}, Kluwer, 2004.


\end{thebibliography}
\end{document}